\documentclass[a4paper,10pt]{amsart}
\usepackage[english]{babel}
\usepackage[utf8]{inputenc}
\usepackage[T1]{fontenc}

\usepackage{amssymb}
\usepackage{mathrsfs}
\usepackage{hyperref}
\usepackage[usenames,dvipsnames]{xcolor}
\hypersetup{colorlinks,%
citecolor=Black,%
filecolor=Black,%
linkcolor=Black,%
urlcolor=Black}
\usepackage{enumitem}	
\newcommand{\N}{\mathbb{N}}	
\newcommand{\R}{\mathbb{R}}	
\newcommand{\rn}{\R^n}
\newcommand{\rem}{\R^m}
\newcommand{\vf}{\varphi}

\newcommand{\dd}{\,\mathrm{d}}	
\newcommand{\eps}{\varepsilon}		

\theoremstyle{plain}
\newtheorem{proposition}{Proposition}
\newtheorem{theorem}[proposition]{Theorem}
\newtheorem{lemma}[proposition]{Lemma}
\newtheorem{corollary}[proposition]{Corollary}

\theoremstyle{definition}
\newtheorem{definition}[proposition]{Definition}
\newtheorem{remark}[proposition]{Remark}
\newtheorem{ex}[proposition]{Example}

\theoremstyle{remark}

\title[ $C^\infty$-approximation  of functions ]{Optimal $C^\infty$-approximation of functions with exponentially or sub-exponentially integrable derivative }

\date{\today}

\author[Ambrosio]{Luigi Ambrosio}
\address[Ambrosio]{Scuola Normale Superiore, 56126 Pisa, Italy}
\email{luigi.ambrosio@sns.it}

\author[Nicolussi Golo]{Sebastiano Nicolussi Golo}
\address[Nicolussi Golo]{Department of Mathematics and Statistics, 40014 University of Jyväskylä, Finland}	
\email{sebastiano.s.nicolussi-golo@jyu.fi}

\author[Serra Cassano]{Francesco Serra Cassano}
\address[Serra Cassano]{Dipartimento di Matematica, Università di Trento, via Sommarive, 14, 38123 Trento, Italy}
\email{francesco.serracassano@unitn.it}

\thanks{
L.A.~and F.S.C.~have been supported by the PRIN 2017 project ``Gradient flows, Optimal Transport and Metric Measure Structures''.
S.N.G.~has been supported by the Academy of Finland (%
grant 328846 ``Singular integrals, harmonic functions, and boundary regularity in Heisenberg groups'',
grant 322898 ``Sub-Riemannian Geometry via  Metric-geometry and Lie-group Theory'',
grant 314172 ``Quantitative rectifiability in Euclidean and non-Euclidean spaces''),
S.N.G.~and F.S.C.~have been supported 
	by the INdAM – GNAMPA Project 2019 ``Rectifiability in Carnot groups''.
Data sharing not applicable to this article as no datasets were generated or analysed during the current study.}
	


\begin{document}
\begin{abstract}
We discuss Meyers-Serrin's type results for smooth approximations of functions $b=b(t,x):\mathbb{R}\times\mathbb{R}^n\to\mathbb{R}^m$, with convergence of an energy of the form
\[
\int_{\mathbb{R}}\int_{\mathbb{R}^n} w(t,x) \varphi\left(|Db(t,x)|\right)\mathrm{d} x \mathrm{d} t\,,
\]
where $w>0$ is a suitable weight function, and $\varphi:[0,\infty)\to [0,\infty)$ is a convex function with $\varphi(0)=0$ having exponential or sub-exponential growth.
\end{abstract}

\maketitle
\tableofcontents

\section{Introduction and results}
In this note we deal with the approximation of functions $b=b(t,x):I\times\Omega\to\rem$
by smooth ones, where $I\subset\R$ and $\Omega\subset\rn$  are an open interval and an open set, respectively.
Our main motivation comes from \cite{ANGSC}, where $m=n$ and $b$ is a possibly nonautonomous vector field. In that paper
we are dealing with {\it a priori} 
upper bounds for Sobolev norms of the flow of $b$ 
by means of
quantities $N_{\vf,w}(|Db|)$ that depend on the spatial derivative
$Db$ of $b$.
The quantities $N_{\vf,w}(|Db|)$ are \emph{energies}
of the form
\begin{equation*}
N_{\vf,w}(|Db|):=\, \int_I\int_\Omega w(t,x) \vf\left(|Db(t,x)|\right)\dd x \dd t\,,
\end{equation*}
where $w>0$ is a suitable weight function (related to ${\rm dist}(x,\partial\Omega)$, 
or to the length of the maximal interval of the ODE associated to $b$) 
and $\vf:[0,\infty)\to [0,\infty)$ is a convex function with $
\vf(0)=0$ having exponential or sub-exponential growth, the model case being 
$\exp_*(t):=\exp(t)-1$. 
When one tries to extend the {\it a priori} estimates from the case of smooth vector fields $b$ 
to those having only a Sobolev spatial regularity, one faces the difficulty of passing the quantity
$N_{\vf,w}$ to the limit. 

In this context, if $\vf$ had polynomial growth, a weighted and $t$-dependent version of the
celebrated Meyers-Serrin Theorem \cite{MS} would be applicable, providing even a smooth approximation $(b_h)_h$ with
$N_{\vf,w}(|D(b_h-b)|)\to 0$. In general, as the discussion below shows, this kind of approximation fails when
$\vf$ does not satisfy a doubling condition. However, we realized that the exponential (or subexponential) case
is a borderline one. Indeed, thanks to the weak subadditivity condition 
\[
\vf(a+b)\le\,k\,[(1+\vf(b))\,\vf(a)+\vf(b)]
\]
we are able to prove convergence of the energy $N_{\vf,w}$ and density of smooth functions with respect to
modular convergence when $\vf$ is strictly convex. This kind of convergence in energy, though weaker 
than convergence with respect to Luxenburg norm (or modular convergence, when $\vf$ is not strictly convex), should be compared with the theory of $BV$ functions, where smooth functions are
not dense in $BV$ norm, but dense in energy.
Moreover, this convergence will be
 sufficient to pass our {\it a priori} estimates to the limit
in the paper \cite{ANGSC}.\\

In the classical setting of Orlicz spaces (see Section~\ref{section2}) the weighted $\vf$-energy $N_{\vf,w}$ is called a 
{\it modular} and we put some of our results in this context.
The main approximation result of the note reads as follows.

\begin{theorem}\label{thm:mainnew}
Let $I\subset\R$ be an open interval and $\Omega\subset\rn$ an open set.
Let $w:\,I\times\Omega\to (0,\infty)$ be a Borel function uniformly bounded from above and from below on compact subsets of $I\times\Omega$. Let $\vf:\,[0,\infty)\to [0,\infty)$ be a convex function satisfying
$\vf(0)=\,0$ and
and for which there exists a positive constant $k_\vf$ such that
\begin{equation}\label{growthvf}
\vf(a+b)\le\,k_\vf[\vf(a)\,\vf(b)+\,\vf(a)+\vf(b)]\qquad\text{for all $a,\,b\in [0,\infty)$}\,.
\end{equation}
Let $b\in L^1_{\rm loc}(I;W^{1,1}_{\rm loc}(\Omega;\rem))\cap C^0(I\times\Omega;\rem)$ satisfy 
\begin{equation}\label{eq5ff48a0aw}
N_{\vf,w}(|Db|)<\infty\,.
\end{equation}
Then there exist $b_h\in C^\infty(I\times\Omega;\rem)$ 
satisfying 
\begin{equation}\label{convW^11loc}
b_h\to b \text{ in }L^1_{\rm loc}(I\times\Omega;\rem), 
\qquad Db_h\to Db\text{ in }L^1_{\rm loc}(I\times\Omega;\R^{nm})\,,\end{equation}
and
\begin{equation}\label{convL1mainthm}
w\vf\left(|Db_h|\right)\to w\vf\left(|Db|\right)\text{ in }L^1(I\times\Omega)\,.
\end{equation}
In particular,
\begin{equation}\label{eq5ff48a67}
\lim_{h\to\infty} N_{\vf,w}(|Db_h|)
= N_{\vf,w}(|Db|) \,.
\end{equation}
\end{theorem}

Besides the model case of $\vf=\exp_*$, we are able to consider the functions
\begin{equation}\label{E*function}
\exp_{\gamma,\tau}^*(t):=\,\exp_{\gamma,\tau}(t)-1 ,
\end{equation}
where $0\le\,\gamma\le\,1$, $\tau>0$ and
\begin{equation}
\label{Efunction2}
\exp_{\gamma,\tau}(t):=\exp\left( \displaystyle\frac{t}{(\log(t+\tau))^\gamma}\right) .
\end{equation}

It is easy to see that a convex function $\vf$ with polynomial growth satisfies \eqref{growthvf}, see Remark \ref{convdoublvf}.
We will show in Lemma~\ref{admisubexpbeh} that $\exp_{\gamma,\tau}^*$ satisfies the conditions in Theorem~\ref{thm:mainnew} for $\vf$ with $k_\vf=1$, if
$\tau$ is sufficiently large.
The functions $\exp_{\gamma,\tau}$, though convex, do not have null derivative at $0$, and therefore do not fit exactly in the theory
of $N$-functions. Therefore, in order to provide a bridge with the theory of $N$-functions of Orlicz spaces, we will also consider the modified functions
\begin{equation}\label{Agammafunction}
\widetilde \exp_{\gamma,\tau}(t):=\,\exp_{\gamma,\tau}(t)-1-\frac{t}{(\log \tau)^\gamma}=
\exp_{\gamma,\tau}^*(t)-\frac{t}{(\log \tau)^\gamma}\, ,
\end{equation}
which are indeed $N$-functions and can be treated it by comparison with $\exp_{\gamma,\tau}^*$.

\begin{corollary}\label{subexpoenergy}  
Let $w:\,I\times\Omega\to (0,\infty)$ be a Borel function
uniformly bounded from above and from below on compact subsets of $I\times\Omega$. 
Let $b\in L^1_{\rm loc}(I;W^{1,1}_{\rm loc}(\Omega;\rem))\cap C^0(I\times\Omega;\rem)$ 
satisfy \eqref{eq5ff48a0aw}
with
\begin{equation}\label{vf=Egamma}
\text{$\vf=\,\widetilde{\exp}_{\gamma,\tau}$ with $\tau$ sufficiently large}
\end{equation}
and
\begin{equation}\label{vf=Agamma}
\text{either $w|Db|\in L^1(I\times\Omega)$, or $w\in L^1(I\times\Omega)$}\,.
\end{equation}
Then there exist $b_h\in C^\infty(I\times\Omega;\rem)$ 
satisfying \eqref{convW^11loc}, \eqref{convL1mainthm} and \eqref{eq5ff48a67}.
\end{corollary}

Notice that, when the functions $\vf$ we are dealing with have a more than polynomial growth at infinity
and the weight function $w$ is uniformly bounded from 0 on compact subsets, Corollary~\ref{subexpoenergy}, the Sobolev Embedding Theorem grants
continuity of $b$ with respect to the spatial variable. 
But,
in the proof of
Theorem~\ref{thm:mainnew}, it seems that the continuity of $b$ with respect to $t$ is also needed 
(cf.~the estimate of the term $z_\delta$).
However, if we assume the
weight $w$ to be time-independent,
we can adapt the proof of Theorem~\ref{thm:mainnew} to drop the continuity assumption on $b$,
and we obtain the following extension of Corollary~\ref{subexpoenergy}.

\begin{theorem}\label{subexpoenergywindepx}
Let $w:\,\Omega\to (0,\infty)$ be  a Borel function
uniformly bounded from above and from below on compact subsets of $\Omega$. 
Assume that either $\vf=\exp^*_{\gamma,\tau}$ or $\vf=\widetilde{\exp_{\gamma,\tau}}$ and \eqref{vf=Agamma} holds, with $\tau$ sufficiently large
and let $b\in L^1_{\rm loc}(I;W^{1,1}_{\rm loc}(\Omega;\rem))$  satisfy \eqref{eq5ff48a0aw}.
Then there exist $b_h\in C^\infty(I\times\Omega;\rem)$ satisfying \eqref{convW^11loc}, \eqref{convL1mainthm} and \eqref{eq5ff48a67}.
\end{theorem}

When the function $\vf$ is strictly convex, from the $\vf$-energy convergence of the Jacobian matrices of autonomous vector fields, one can obtain the {\it modular convergence} (see Definition~\ref{def:conv}).

\begin{theorem} \label{thm:main2}
Let $\vf:[0,\infty)\to [0,\infty)$ be a strictly convex function, and assume that
$b_h\to b$ in $L^1_{\rm loc}(\Omega;\rem)$ with
\begin{equation}
\int_\Omega\vf(|Db_h|)\dd x\to \int_\Omega\vf(|Db|)\dd x<\infty\,. \end{equation}
Then
\begin{equation}
\int_\Omega\vf\left(\frac{|Db_h-D b|}{2}\right)\dd x\to 0\,.
\end{equation}
\end{theorem}

Finally, if $\vf=\,\widetilde \exp_{\gamma,\tau}$, being $\widetilde \exp_{\gamma,\tau}$ a $N$-function (see Lemma~\ref{admisubexpbeh}.(ii)), we can set our result in the classical setting of Orlicz-Sobolev spaces (see Section~\ref{section2}). Therefore, an immediate consequence of  Theorem \ref{subexpoenergywindepx} is the following approximation result in the Orlicz-Sobolev class $W^1 K_{\widetilde \exp_{\gamma,\tau}}(\Omega)$.

\begin{theorem}\label{approxOS} 
Let $\vf=\,\widetilde \exp_{\gamma,\tau}$ be the $N$-function in \eqref{Agammafunction} with $\tau$ given by Lemma~\ref{admisubexpbeh} and $u\in W^1 K_{\widetilde \exp_{\gamma,\tau}}(\Omega)$. 
Suppose that
\begin{equation}
\text{either $|Du|\in L^1(\Omega)$, or $\Omega$ has finite measure\,. }
\end{equation}
Then there exists $(u_h)_h\subset C^\infty(\Omega)\cap W^1 K_{\widetilde \exp_{\gamma,\tau}}(\Omega)$ such that $(u_h)_h$ is mean convergent to $u$ (with respect to the modular $N_{\widetilde \exp_{\gamma,\tau}}$) and $(|Du_h|)_h$ is ${\widetilde \exp_{\gamma,\tau}}$-energy convergent to $|Du|$, that is, 
\begin{equation}
\lim_{h\to\infty}N_{\widetilde \exp_{\gamma,\tau}}\left(u_h-u\right)
=\,0\text{ and }\lim_{h\to\infty}N_{\widetilde \exp_{\gamma,\tau}}\left(|Du_h|\right)
= N_{\widetilde \exp_{\gamma,\tau}}\left(|Du|\right)\,.
\end{equation}
\end{theorem}
To our knowledge, the previous result does not seem be a consequence of the well-known results about approximation by smooth functions in Orlicz-Sobolev spaces (see Section~\ref{densityOS}), even in the classical case with $\gamma=0$.

\section{Recalls of  some density results of smooth functions in Orlicz and Orlicz-Sobolev spaces}
\label{section2}
We will quickly recall here the notions of Orlicz and Orlicz-Sobolev spaces and some their main properties. In particular, we will focus on  the main density results of smooth functions in Orlicz and Orlicz-Sobolev spaces. We will mainly use the notation  from~\cite[Ch.~VIII]{A}. 

\subsection{N-functions}\label{Nfunct} 
A function $\vf:[0,\infty)\to[0,\infty)$
is called a {\it $N$-function}, if 
\[
\vf(t):=\,\int_0^t a(s)\,\dd s\text{ if }t\ge\,0,
\]
with $a:\,[0,\infty)\to [0,\infty)$ satisfying:
\begin{itemize}
\item $a(0)=\,0$, $a(t)>\,0$ if $t>\,0$, and $\lim_{t\to\infty}a(t)=\,\infty$;
\item $a$ is nondecreasing, that is, if $t\ge\,s\ge\,0$, then $a(t)\ge\,a(s)$;
\item $a$ is right continuous, that is, if $t\ge\,0$, then $\lim_{s\to t^+}a(s)=\,a(t)$.
\end{itemize}

Given a $N$-function $\vf$ and $\lambda>\,0$, we denote by $\vf_\lambda:[0,\infty)\to [0,\infty)$ the function
\[
\vf_\lambda(t):=\,\vf\left(\frac{t}{\lambda}\right)\text{ if }t\ge\,0\,,
\]
which is still a $N$-function.

A function  $\vf$ is said to satisfy   a {\it global $\Delta_2$-condition }  if there exists $k>\,0$ such that
\[
\vf(2t)\le\,k\,\vf(t)\text{ for each }t\ge\,0\,.
\]
A function  $\vf$ is said to satisfy   a {\it  $\Delta_2$-condition near infinity } if there exist $k,t_0>\,0$ such that
\[
\vf(2t)\le\,k\,\vf(t)\text{ for each }t\ge\,t_0\,.
\]

\begin{remark}\label{convdoublvf} 
Observe that a convex function $\vf$ satisfying a global $\Delta_2$-condition trivially fulfills condition \eqref{growthvf}. 
Indeed, by the convexity and $\Delta_2$-condition, we can get the following estimate
\[
\vf(a+b)\le\,\frac{1}{2}\vf(2a)+ \frac{1}{2}\vf(2b)\le\,\frac{k}{2}\left(\vf(a)+\vf(b)\right)\text{ for each }a,\,b\in\R\,.
\]
\end{remark}

Given $\Omega\subset\rn$ and a N-function $\vf$ , a pair $(\vf,\Omega)$ is said to be  {\it $\Delta$-regular} if
\begin{itemize}
\item $\vf$ satisfies a {global $\Delta_2$-condition}, 

\noindent or

\item $\vf$ satisfies a {$\Delta_2$-condition near infinity}  and $\Omega$ has finite measure.
\end{itemize}

\subsection{The Orlicz class \texorpdfstring{$K_\vf(\Omega)$}{Kphi(Omega)}}

Let $\Omega\subset\rn$ be an open set  and let $\vf$ be  a $N$-function. The {\it Orlicz class } $K_\vf(\Omega)$ is the set of all (equivalence classes modulo equality a.e.~on $\Omega$ of) measurable functions $u:\Omega\to\R$ such that
\[
N_\vf(u):=\,\int_\Omega \vf(|u(x)|)\,\dd x<\,\infty\,.
\]
In the theory of modular spaces, the map $u\mapsto N_\vf(u)$ is called a {\it modular} (\cite[pg. 82]{RR}). 
A comprehensive account of modular function spaces  can be found in \cite{K}. 
We treat the case of real-valued functions for simplicity, but
all results have an obvious extension to the case of $\R^m$-valued maps.

Let us recall some properties of the Orlicz class $K_\vf(\Omega)$.

\begin{proposition}\label{propKA}
Given an open set $\Omega\subset\rn$ and a $N$-function $\vf$, 
the following statements hold:
\begin{itemize}
\item [(i)] $K_\vf(\Omega)$ is a convex set of measurable functions.
\item [(ii)] $K_{\vf_\lambda}(\Omega)\supseteq K_\vf(\Omega)$ if $\lambda\ge\,1$ and $K_{\vf_\lambda}(\Omega)\subseteq K_\vf(\Omega)$ if $\lambda\le\,1$, 
where $\vf_\lambda(t):=\vf(t/\lambda)$ is a $N$-function for all $\lambda>0$.
\item [(iii)] If $f,\,g\in K_{\vf}(\Omega)$, then $f+g\in  K_{\vf_2}(\Omega)$ and
\[
N_{\vf_2}(f+g)\le\, \frac{1}{2} N_\vf(f)+\frac{1}{2} N_\vf(g)\,.
\]
\item [(iv)] If $f\in K_{\vf}(\Omega)$ and $\lambda>\,0$, then $\lambda f\in K_{\vf_\lambda}(\Omega)$.
\item [(v)]  
If $\Omega$ has finite measure,
then 
\[
L^\infty(\Omega)\subset K_\vf(\Omega) \subsetneq L^1(\Omega) .
\]
\item[(vi)]
If $\Omega$ has finite measure,
then for every $u\in L^1(\Omega)$ there is a $N$-function $\vf$ such that $u\in K_\vf(\Omega)$.
\end{itemize}
\end{proposition}
\begin{proof} Properties (i), (ii), (iii) and (iv) are immediate consequences of the definition of $K_\vf(\Omega)$ and the convexity of $\vf$.
For the proof of properties (v) and (vi) see, for instance,  \cite{KR}.
\end{proof}

\begin{lemma}[{\cite[Lem.~8.8]{A} or \cite[Ch.~III, Th.~8.2]{KR}}]\label{KAsv}
$K_\vf(\Omega)$ is a vector space if and only if   $(\vf,\Omega)$ is {\it $\Delta$-regular} .
\end{lemma}

\subsection{The Orlicz space \texorpdfstring{$L_\vf(\Omega)$}{Lphi(Omega)}}

The {\it  Orlicz space} $L_\vf(\Omega)$ is defined to be the linear hull of the Orlicz class  $K_\vf(\Omega)$, that is   the smallest vector subspace of $L^1_{\rm loc}(\Omega)$ containing $K_\vf(\Omega)$. It is easy to see that, since $K_\vf(\Omega)$ is convex, one has
\[
L_\vf(\Omega):=\,\left\{\lambda\,u:\,\lambda\in\R,\,u\in K_\vf(\Omega)\right\}\,.
\]
Moreover, from Lemma \ref{KAsv}, $K_\vf(\Omega)=L_\vf(\Omega)$ if and only if $(\vf,\Omega)$ is {\it $\Delta$-regular}.

We can endow $L_\vf(\Omega)$ with the following norm, called {\it Luxemburg norm},
\[
\|u\|_\vf=\, \|u\|_{\vf,\Omega}:=\,\inf\left\{\lambda>\,0:\,\int_\Omega \vf\left(\frac{|u(x)|}{\lambda}\right)\,\dd x\le\,1\right\}\,.
\]
\begin{theorem}[{\cite[Thm.~8.10]{A}}]
 $(L_\vf(\Omega),\|\cdot\|_\vf)$ is a Banach space.
\end{theorem}

\subsection{Convergences in \texorpdfstring{$L_\vf(\Omega)$}{Lphi(Omega)}}
The typical convergences that apply in Orlicz spaces are the following.

\begin{definition}\label{def:conv}
		A sequence of functions $(u_h)_h\subset L_\vf(\Omega)$  is said to be {\it norm convergent}  to $u\in L_\vf(\Omega)$ if
		\[
		\|u_h-u\|_\vf\to 0\text{ as }h\to\infty\,.
		\]
		A sequence of functions $(u_h)_h\subset L_\vf(\Omega)$  is said to be {\it modular convergent}  to $u\in L_\vf(\Omega)$ if there exists $\lambda>\,0$  such that
		\begin{equation}\label{modconv}
		N_\vf\left(\frac{u_h-u}{\lambda}\right)\to 0\text{ as }h\to\infty.
		\end{equation}
		If $\lambda=1$ in \eqref{modconv}, $(u_h)_h$  is said to be {\it mean convergent}  to $u\in L_\vf(\Omega)$. 
		A sequence of functions $(u_h)_h\subset K_\vf(\Omega)$  is said to be {\it $\vf$-energy convergent}  to $u\in K_\vf(\Omega)$ if
		\begin{equation}\label{strictconv}
		N_\vf(u_h)\to  N_\vf(u)\text{ as }h\to\infty.
		\end{equation}
\end{definition}

Norm and modular convergences are  classical in the theory of 
Orlicz spaces  (see, for instance, \cite{A,KR}).   We do not know whether the $\vf$-energy convergence has been already named in the literature.  

The following implications between norm, mean, modular and $\vf$-energy convergence hold.
\begin{proposition}\label{comparconv} Let $(u_h)_h$ and $u$ be in $ L_\vf(\Omega)$.
\begin{itemize}
\item[(i)] Suppose that $(u_h)_h$ is norm convergent to  $u$. Then it is also mean convergent. The converse implication in general
does not hold. It holds if $(\vf,\Omega)$ is {\it $\Delta$-regular}.
\item[(ii)] Suppose that $(\vf,\Omega)$ is {\it $\Delta$-regular}, 
$\vf$ is strictly convex,
$u_h\to u$ a.e.~in $\Omega$ and  $(u_h)_h$ is $\vf$-energy convergent to~$u$.
Then $(u_h)_h$ is norm convergent to  $u$.
\item[(iii)] $(u_h)_h$ is norm convergent to  $u$ if and only if, for each $\lambda>\,0$,
\[
N_\vf\left(\frac{u_h-u}{\lambda}\right)\to 0\text{ as }h\to\infty\,.
\]
\item[(iv)] 
Suppose that $(\vf,\Omega)$ is {\it $\Delta$-regular} and $(u_h)_h\subset K_\vf(\Omega)$ is mean convergent to $u\in K_\vf(\Omega)$. Then $(u_h)_h$ is $\vf$-energy convergent to $u$.
\item[(v)] 
Suppose that  $(2u_h)_h\subset K_{\vf}(\Omega)$ is mean convergent to $2u\in K_{\vf}(\Omega)$ (with respect to the modular $N_{\vf}$). Then  $(u_h)_h\subset K_{\vf}(\Omega)$,  $u\in K_{\vf}(\Omega)$ and  $(u_h)_h$ is $\vf$-energy convergent to $u$.
\item[(vi)] 
Suppose that  $2u\in K_\vf(\Omega)$ and $(u_h)_h$ is norm convergent to  $u$. Then $u_h\in K_\vf(\Omega)$ for $h$ large and    $(u_h)$  is also $\vf$-energy convergent to $u$.
\end{itemize}
\end{proposition}
\begin{proof}
 (i) and (ii) are proven in \cite[Chap.~III, Sect.~3.4, Thm.~12]{RR}.
 The proof of (iii) is somehow elementary, see for instance \cite[Lem.~2.7]{AFY} and \cite[pg.~4]{K}.

We prove (iv) when $\vf$ satisfies a global $\Delta_2$-condition:
in the other case, when $\Omega$ has finite measure and the $\vf$ satisfies a $\Delta_2$-condition near infinity, has a similar proof.
From the assumptions, the sequence $(\vf(|u_h-u|))_h\subset L^1(\Omega)$ converges in $L^1(\Omega)$ to $0$.
 Thus, up to a subsequence, we can assume that
\begin{equation*}
\vf(|u_h-u|)\to 0\text{ a.e. in } \Omega,\text{ as $h\to\infty$}\,.
\end{equation*}
Since $\vf$ is a $N$-function, then $\vf:\,[0,\infty)\to [0,\infty)$ is bijective and $\vf^{-1}:\,[0,\infty)\to [0,\infty)$ is still continuous.
Thus, we also get that
\begin{equation}\label{pointconv}
|u_h-u|=\,\vf^{-1}(\vf(|u_h-u|))\to \vf^{-1}(\vf(0))=\,0\text{ a.e.~in } \Omega,\text{ as $h\to\infty$}\,.
\end{equation}
From the convexity  of $\vf$ and the global $\Delta_2$-condition of $\vf$, it follows that
\begin{equation}\label{domconv2}
\begin{aligned}
\vf(|u_h|)
&\le\,\frac{1}{2}\vf(2|u_h-u|)+ \frac{1}{2}\vf(2|u|)\,\\
&\le\,\frac{k}{2}\left(\vf(|u_h-u|)+\vf(|u|)\right) .
\end{aligned}
\end{equation}
By \eqref{pointconv} and \eqref{domconv2}, we can apply Vitali's convergence theorem and then
\[
\vf(|u_h|)\to \vf(|u|)\text{ in }L^1(\Omega)\text{, as }h\to \infty\,.
\]
Thus \eqref{strictconv} follows.

For (v), we get at once that $(u_h)_h\subset K_{\vf}(\Omega)$ and   $u\in K_{\vf}(\Omega)$, because $\vf$ is increasing.
We can show \eqref{pointconv} as in the proof of claim (iv),
and the convexity of $\vf$ implies
\[
\vf(|u_h|)
\le\,\frac{1}{2}\vf(2|u_h-u|)+ \frac{1}{2}\vf(2|u|)
\]
Thus, applying Vitali's convergence theorem, we still get \eqref{strictconv}.

Finally, we prove (vi).
From the norm convergence and (iii), we can infer that, up to a subsequence,
\[
\vf(|u_h-u|)\to 0\text{  a.e. in }\Omega\text{, as }h\to\infty\,,
\]
and
\[
\vf(2|u_h-u|)\to 0\text{ in } L^1(\Omega)\text{, as }h\to\infty\,.
\]
We can show again \eqref{pointconv} as in claim (iv) and 
$\vf(|u_h|)\le\frac{1}{2}\vf(2|u_h-u|)+ \frac{1}{2}\vf(2|u|)$
from the convexity of $\vf$.
 Thus, applying Vitali's convergence theorem, we get~\eqref{strictconv}.
\end{proof}

\begin{ex}\label{ex}
From items (iv) and (v) of Proposition~\ref{comparconv}, 
one could get the wrong impression that mean convergence implies $\vf$-energy convergence.
We show that this is not the case if $(\vf,\Omega)$ is not
$\Delta$-regular:
for $\Omega=\,(0,1)$ and $\vf=\,\widetilde\exp_0$ (cf.~\eqref{Agammafunction}), we give a sequence of functions $u_h\in K_\vf((0,1))$ that is mean convergent to $u\in K_\vf((0,1))$, but that is not $\vf$-energy convergent.

Let $f_h,f:(0,1)\to\R$ be the functions
\[
\displaystyle{
f_h(x) := 
\begin{cases}
\frac{2\sqrt h}{\log h}&\text{ if }0<\,x<\,\frac{1}{h}\\
\frac{1}{\log h}\frac{1}{\sqrt x}&\text{ if }\frac{1}{h}\le\,x<\,1
\end{cases}
},\quad f(x) :=\frac1{\sqrt{x}} .
\]
Direct computations show that
\begin{align*}
\int_0^1 f(x) \dd x &= 2 , &
\int_0^1 \log(f(x)) \dd x &=  \frac12 , \\
\int_0^1 f_h(x) \dd x &\overset{h\to\infty}{\longrightarrow}0 , &
\int_0^1 \log(1+f_h) \dd x &\overset{h\to\infty}{\longrightarrow}0 , \\
\int_0^1 f(x)\,f_h(x) \dd x &= \frac{4}{\log(h)} + 1 \overset{h\to\infty}{\longrightarrow}1 .\\
\end{align*}
Define
\[
u := \log(f)
\quad\text{ and }\quad 
u_h := \log(f) + \log(1+f_h) .
\]
Then, for $\vf(s) = \exp(s)-1-s$, we have
\begin{align*}
N_\vf(u) 
	&= \int_0^1 f(x) \dd x - 1 - \int_0^1 \log(f(x)) \dd x , \\
N_\vf(u_h) 
	&= \int_0^1 f(x) \dd x - 1 - \int_0^1 \log(f(x)) \dd x \\
		&\qquad	+ \int_0^1 f(x)\,f_h(x) \dd x - \int_0^1 \log(1+f_h) \dd x , \\
N_\vf(u_h-u) 
	&= \int_0^1 f_h(x) \dd x - \int_0^1 \log(1+f_h(x)) \dd x .
\end{align*}
We conclude that $u,u_h \in K_{\vf}((0,1))$, $N_\vf(|u-u_h|)\to 0$ but
\[
N_\vf(|u_h|) - N_\vf(|u|) 
= \int_0^1 f(x)\,f_h(x) \dd x - \int_0^1 \log(1+f_h) \dd x
\overset{h\to\infty}{\longrightarrow}1 ,
\]
that is, $u_h$ is not $\vf$-energy convergent to $u$.
Notice that the key fact is that $f_h\to0$ in $L^1((0,1))$ but $f\cdot f_h\not\to0$.
Let us also observe that $(2u_h)_h\subset K_{\vf}(\Omega)$, but neither $2u\in K_{\vf}(\Omega)$ nor  $(2u_h)_h$ is mean convergent to $2u$ with respect to  $N_{\vf}$. 
\end{ex}

\subsection{The vector space \texorpdfstring{$E_\vf(\Omega)$}{Ephi(Omega)}}
Let $E_\vf(\Omega)$ denote the closure in $(L_\vf(\Omega),\|\cdot\|_\vf)$ of the space of functions $u$ which are bounded 
in $\Omega$ with bounded support in $\Omega$.

One can see  (\cite[Sect.~8.14]{A}) that $E_\vf(\Omega)\subset\,K_\vf(\Omega)$ and that, if $(\vf,\Omega)$ is {\it $\Delta$-regular}, then
\[
E_\vf(\Omega)=\,K_\vf(\Omega)=\,L_\vf(\Omega)\,.
\]
Moreover the following characterization of $E_\vf(\Omega)$ holds ([A, Lemma 8.15]).
\begin{lemma}\label{EAmaxsubsp} $E_\vf(\Omega)$ is  the maximal linear subspace of $K_\vf(\Omega)$.
\end{lemma}
\begin{corollary}\label{cor620a76f2}
If  $(\vf,\Omega)$ is not {\it $\Delta$-regular},  it holds that
\[
E_\vf(\Omega)\subsetneq\,K_\vf(\Omega) \subsetneq L_\vf(\Omega)\,.
\]
\end{corollary}
\begin{proof}
By Lemma \ref{KAsv}, $K_\vf(\Omega)$ cannot be a vector space. Thus, by Lemma \ref{EAmaxsubsp}, we get the desired conclusions.
\end{proof}
Let us now recall some density results in $(E_\vf(\Omega),\|\cdot\|_A)$.
\begin{theorem}[{\cite[Thm. 8.20]{A}}]\label{densEA}
Let $\Omega\subset\rn$ be an open set  and let $\vf$ be  a $N$-function.
\begin{itemize}
\item[(i)] $C^\infty_c(\Omega)$ are dense in $(E_\vf(\Omega),\|\cdot\|_\vf)$.
\item[(ii)] $(E_\vf(\Omega),\|\cdot\|_\vf)$ is separable.
\item[(iii)] Let us extend $u\in E_\vf(\Omega)$ to the whole $\rn$ so as to vanish outside $\Omega$ and 
let $(\rho_\eps)_\eps$ be a family of mollifiers on $\rn$. Then
\[ 
\rho_\eps*u\to u \text{ in }(E_\vf(\Omega),\|\cdot\|_\vf)\text{, as } \eps\to 0\,.
\]
\end{itemize}
\end{theorem}
An immediate consequence of Theorem \ref{densEA}
is that,
if  $(\vf,\Omega)$ is not $\Delta$-regular, then $C^0_c(\Omega)$ is not dense in $(L_\vf(\Omega),\|\cdot\|_\vf)$.
In fact, one can prove the following stronger result:
\begin{theorem}[{\cite[Chap.~II, Thm.~10.2]{KR}}]
If the pair $(\vf,\Omega)$ is not $\Delta$-regular, then \mbox{$(L_\vf(\Omega),\|\cdot\|_\vf)$} is not separable.
\end{theorem}

Let us also point out  some density results  in $K_\vf(\Omega)$ with respect to the modular convergence.

\begin{theorem}
Let $\Omega\subset\rn$ be an open set  and let $\vf$ be  a $N$-function.
\begin{itemize}
\item[(i)]
The set of bounded functions on $\Omega$ contained in $K_\vf(\Omega)$  
with bounded support
is dense in $K_\vf(\Omega)$ with respect to the mean convergence, that is, for each $u\in K_\vf(\Omega)$ there exists a sequence of bounded functions $(u_h)_h\subset K_\vf(\Omega)$ such that 
\[
N_\vf(u_h-u)\to 0\text{, as }h\to \infty\,.
\]
\item[(ii)] $C^0_c(\Omega)$  is dense in $K_\vf(\Omega)$ with respect to the modular convergence with $\lambda=4$.
More precisely, for each $u\in K_\vf(\Omega)$, there is a sequence $(u_h)_h\subset C^0_c(\Omega)$ such that
\[
N_\vf\left(\frac{u_h-u}{4}\right)\to 0\text{, as }h\to\infty\,.
\]
\end{itemize}
\end{theorem}

\begin{proof} 
The proof of part (i) can be found in \cite[Chap.~II, pg.~77]{KR} or \cite[Sect.~8.14]{A}.

To prove part (ii), 
given $u\in K_\vf(\Omega)$ and $\epsilon>0$,
we first notice that, by a standard truncation argument in $\Omega$,
there is a function $\tilde{u}\in K_\vf(\Omega)$ with  support compactly contained in $\Omega$
 and with $N_{\vf}(u-\tilde{u})<\epsilon$.

Next, let $\tilde{u}_k:=\max\{-k,\min\{\tilde{u},k\}\}$ be the standard truncation of $u$,  $F_k:=\{|\tilde{u}|>k\}$ and  $f\in C^0_c(\Omega)$ with $\sup|f|\le k$. 
We estimate
\begin{align*}
N_{\vf_2}(\tilde{u}-f)
&= \int_{\Omega\setminus F_k}\vf\left(\frac{|\tilde{u}_k-f|}{2}\right)\dd x+\int_{F_k} \vf\left(\frac{|\tilde{u}-f|}{2}\right)\dd x \\
&\le \int_{\Omega\setminus F_k}\vf\left(\frac{|\tilde{u}_k-f|}{2}\right)\dd x
	+ \frac 12\int_{F_k} \vf(|\tilde{u}|)\dd x
	+ \frac 12\int_{F_k} \vf(|f|)\dd x \\
&\le \int_{\Omega\setminus F_k}\vf\left(\frac{|\tilde{u}_k-f|}{2}\right)\dd x
	+ \int_{F_k} \vf(|\tilde{u}|)\dd x .
\end{align*}
Now, since $\int_\Omega \vf(|\tilde{u}|)\dd x<\infty$, 
we can choose $k$ so large that the second integral is smaller than $\epsilon/2$.
Since $\tilde{u}_k$ has compact support and thanks to Lusin's theorem, we can find
 $f\in C^0_c(\Omega)$ with $|f|\le k$ and the Lebesgue measure of
$\{x\in\Omega:\ \tilde{u}_k(x)\neq f(x)\}$ sufficiently small, in such a way that also the 
first integral gives a contribution smaller than $\epsilon/2$. 

In conclusion, for every $\epsilon>0$ we have $f\in C^0_c(\Omega)$ such that 
\[
N_{\vf_4}(u-f) 
\le \frac12 (N_{\vf_2}(u-\tilde u) + N_{\vf_2}(\tilde u - f) )
\le \frac12 (N_{\vf}(u-\tilde u)/2 + \epsilon )
\le \epsilon .
\] 
\end{proof}
\subsection{ Orlicz-Sobolev spaces and density results of smooth functions.}\label{densityOS}
Given a $N$-function $\vf$, the {\it Orlicz-Sobolev vector space $W^1L_\vf(\Omega)$} consists of those (equivalence classes of) functions $u\in L_\vf(\Omega)\cap W^{1,1}_{loc}(\Omega)$ whose weak derivatives $D_iu\in L_\vf(\Omega)$ for each $i=1,\ldots,n$. The vector space $W^1E_\vf(\Omega)$ and the convex set $W^1K_\vf(\Omega)$  are defined in analogous fashion. Obviously
\[
W^1E_\vf(\Omega)\subset W^1K_\vf(\Omega)\subset W^1L_\vf(\Omega)\,.
\]

It is easy to see (see, for instance, \cite[\S8.27]{A}) that  $W^1L_\vf(\Omega)$ is a Banach space with respect to the norm
\[
\|u\|_{1,\vf}:=\,\max\{\|u\|_{\vf},\|D_1u\|_{\vf},\ldots, \|D_nu\|_{\vf}\} \,.
\]
Notice also that, since
\[
\max_i|D_iu| 
	\le |Du|
	\le\,\sum_{i=1}^n|D_iu|
	\text{ a.e. in }\Omega\,,
\]
and $L_\vf(\Omega)$ is a linear space, an equivalent norm on $W^1L_\vf(\Omega)$ is given by.
\[
\|u\|_{\vf}+\||Du|\|_{\vf}\,.
\]
Observe that $W^1E_\vf(\Omega)$  turns out to be a closed subspace  of $W^1L_\vf(\Omega)$. Moreover $W^1E_\vf(\Omega)$ coincides with $W^1L_\vf(\Omega)$ if and only if $(\vf,\Omega)$ is  {\it $\Delta$-regular}. 
 Notice also that, for the applications we have in mind, what is more relevant is the $\vf$-integrability of the derivative, rather than the integrability of the function which, also in view of Sobolev embeddings,
could be qualified in a different way, see also Remark~\ref{rem:different_integrability}.

 Celebrated Meyers-Serrin's result was extended from the classical Sobolev spaces  to the Orlicz-Sobolev space $W^1E_\vf(\Omega)$ in \cite{DT} (see also \cite{AFY}).
 \begin{theorem}[{\cite{DT}}]\label{MSnorm}
 $C^\infty(\Omega)\cap W^1E_\vf(\Omega)$ is dense in $(W^1E_\vf(\Omega),\|\cdot\|_{1,\vf}) $.
 \end{theorem}
It is easy to see  that the previous result also fails  for  functions in the  Orlicz-Sobolev class $W^1K_\vf(\Omega)$, and so also in the Orlicz-Sobolev space $W^1L_\vf(\Omega)$, provided that $(\vf,\Omega)$ is not $\Delta$-regular, as the following example shows. 

\begin{ex}\label{notMSKvf }
Assume that $n=1$,  $\Omega=\,(-1,1)$, let $\vf=\,\widetilde\exp_0$ be $N$-function in \eqref{Agammafunction} with $\gamma=0$ and let
\[
u(x):=
\displaystyle{\begin{cases}
\frac{x}{2}\,\log\frac{1}{e|x|}&\text{ if }|x|\le\,\frac{1}{e}\\
0&\text{ if }\frac{1}{e}<\,|x|\,<\,1
\end{cases}
}
\,.
\]
Then it is easy to see that $u\in W^1K_\vf(\Omega)\setminus W^1E_\vf(\Omega)$, since the weak derivative
\[
u'(x):=
\begin{cases}
\log\displaystyle{\frac{1}{e\,\sqrt {|x|}}}&\text{ if }|x|<\,\frac{1}{e}\\
0&\text{ if }\frac{1}{e}<\,|x|\,<\,1
\end{cases}
\text{ a.e. }x\in\Omega\,
\]
belongs to $K_\vf(\Omega)\setminus E_\vf(\Omega)$. Indeed
\[
\int_{-1}^1\vf(|u'|)\,\dd x=\int_{-1}^1\left(\exp(|u'|)-|u'|-1 \right)\,\dd x<\,\infty\,,
\]
but $2|u'|\notin K_\vf(\Omega)$, since
\[
\int_{-1}^1\vf(2|u'|)\,\dd x=\int_{-1}^1\left(\exp(2|u'|)-2|u'|-1 \right)\,\dd x=\,\infty\,.
\]
Thus $|u'|\notin E_\vf(\Omega)$, since  $E_\vf(\Omega)$ is  a linear subspace .
By contradiction, assume there exists a sequence $(u_h)_h\subset C^\infty(\Omega)\cap W^1L_\vf(\Omega)$ such that $u_h\to u$ in  $W^1L_\vf(\Omega)$, as $h\to\infty$. In particular, it also follows that
\begin{equation}\label{convderuh}
u'_h\to u' \text{ in } L_\vf(\Omega)\text{ as }h\to\infty\,.
\end{equation}
Let $\psi\in C^0_c(\Omega)$ such that $0\le\,\psi\le\,1$ and $\psi\equiv\,1$ in $(-1/e, 1/e)$ and let 
\[
v_h:=\,\psi\,u'_h\,.
\]
By Proposition~\ref{comparconv}(iii) and \eqref{convderuh}, it still holds that
\[
E_\vf\ni\psi\,u'_h\to \psi\,u'=\,u'\text{ in } L_\vf(\Omega)\text{, as }h\to\infty\,.
\]
Then a contradiction since $u'\notin E_\vf$.
\end{ex}

A weaker density result of regular functions in $W^1L_\vf(\Omega)$ holds by using the modular convergence, as shown in~\cite{G}.
\begin{theorem}[{\cite{G}}]\label{MSmod}
 Let $u\in W^1L_\vf(\Omega)$. Then there exist $\lambda>\,0$ and a sequence of functions $(u_h)_h\subset C^\infty(\Omega)\cap W^1L_\vf(\Omega)$ such that
 \[
 N_\vf\left(\frac{u_h-u}{\lambda}\right)\to 0\text{ and }N_\vf\left(\frac{D_iu_h-D_iu}{\lambda}\right)\to 0\text{ as }h\to\infty\,,
 \]
 for each $i=1,\ldots,n$. In particular it suffices to choose $\lambda$  such that $\frac{16}{\lambda}D_iu\in K_\vf(\Omega)$.
\end{theorem}

\section{Exponential and sub-exponential N-functions}
It is easy to see that a convex function $\vf$ with polynomial growth satisfies \eqref{growthvf}, see Remark \ref{convdoublvf}.
We will show in Lemma~\ref{admisubexpbeh} that $\exp_{\gamma,\tau}^*$ satisfies the conditions in Theorem~\ref{thm:mainnew} for $\vf$ with $k_\vf=1$, if
$\tau$ is sufficiently large.

Recall from~\eqref{E*function} and~\eqref{Efunction2} that we have set
\[
\exp_{\gamma,\tau}(t):=\exp\left( \displaystyle\frac{t}{(\log(t+\tau))^\gamma}\right) ,
\quad\text{and}\quad
\exp_{\gamma,\tau}^*(t):=\,\exp_{\gamma,\tau}(t)-1  .
\]
The functions $\exp_{\gamma,\tau}^*$, though convex, do not have null derivative at $0$, and therefore do not fit exactly in the theory
of $N$-functions. Therefore, in order to provide a bridge with the theory of $N$-functions of Orlicz spaces, we will also consider the modified functions
\begin{equation*}
\widetilde \exp_{\gamma,\tau}(t):=\,\exp_{\gamma,\tau}(t)-1-\frac{t}{(\log \tau)^\gamma}=
\exp_{\gamma,\tau}^*(t)-\frac{t}{(\log \tau)^\gamma}\, ,
\end{equation*}
which are indeed $N$-functions and can be treated by comparison with $\exp_{\gamma,\tau}^*$.

\begin{lemma}\label{admisubexpbeh}  
There exists $\tau_0>0$ such that, for all $\tau\ge\tau_0$ and $0\le\gamma\le 1$, one has
\begin{itemize}
\item[(i)]
$\exp_{\gamma,\tau}$ is a smooth strictly convex increasing function. 
Moreover, for all $t,s\in[0,\infty)$, 
\begin{equation}\label{sublineartEgamma}
\exp_{\gamma,\tau}(t+s)\le\,\exp_{\gamma,\tau}(t)\,\exp_{\gamma,\tau}(s) ,
\end{equation}
and $\exp_{\gamma,\tau}^*$ satisfies~\eqref{growthvf} with $k_\vf=1$, that is,
\begin{equation}\label{sublineartE*gamma}
\exp_{\gamma,\tau}^*(t+s)\le\,\exp_{\gamma,\tau}^*(t)\,\exp_{\gamma,\tau}^*(s)+\exp_{\gamma,\tau}^*(t)+\exp_{\gamma,\tau}^*(s) .
\end{equation}
\item[(ii)] $\widetilde \exp_{\gamma,\tau}$ is a $N$-function satisfying
\begin{equation}\label{propAgamma}
\widetilde \exp_{\gamma,\tau}(t)\le\,\exp_{\gamma,\tau}^*(t)\text{ if }t\ge\,0\text{ and }\lim_{t\to\infty}\frac{\widetilde \exp_{\gamma,\tau}(t)}{\exp_{\gamma,\tau}^*(t)}=\,1\,.
\end{equation}
\end{itemize}
\end{lemma}

\begin{proof}
\noindent(i)  By a simple calculation, it is easy to see that, if $\tau>\,1$, $\exp_{\gamma,\tau}$ is well-defined, $\exp_{\gamma,\tau}\in C^\infty([0,\infty))$ and 
\begin{align*}
	\exp'_{\gamma,\tau}(t)
	&=\,\exp\left( \frac{t}{(\log( t+\tau))^\gamma}\right)
	\frac{\log( t+\tau)-\displaystyle{\frac{\gamma\,t}{t+\tau}}}{(\log (t+\tau))^{\gamma+1}}, \\
	\exp''_{\gamma,\tau}(t)
	&=\,\frac{\exp\left( \displaystyle{\frac{t}{(\log (t+\tau))^\gamma}}\right)}{(\log (t+\tau))^{2\gamma+2}}\Big[\left(\log (t+\tau)-\frac{\gamma\,t}{t+\tau}\right)^2\\
	&\qquad\qquad-(\log (t+\tau))^{\gamma+1}\frac{\gamma\,\tau+ \gamma (t+\tau)}{(t+\tau)^2}+\frac{\gamma (\gamma+1)t (\log (t+\tau))^\gamma}{(t+\tau)^2}\Big] 
\end{align*}
for all $t\ge 0$. Now, observe that, if $\tau>e$ and $t\ge0$, then
\begin{equation}\label{estimfirstdertildeEgamma}
\frac{\gamma\,t}{t+\tau}\le\,\gamma,
\qquad\frac{\gamma\,\tau+\gamma (t+\tau)}{(t+\tau)^2}\le\,\frac{ 2\,\gamma}{\tau},
\text{ and }
\log(t+\tau)> 1\ge\gamma.
\end{equation}
Combining these inequalities, it follows that, if $\tau>e$,
\begin{equation}\label{derEgammapos}
\exp'_{\gamma,\tau}(t)>0\text{ for each }t\ge\,0\,,
\end{equation}
so that $\exp_{\gamma,\tau}$ is strictly increasing on $[0,\infty)$.
  Let us now show that,  for sufficiently
large $\tau$, one has
\begin{equation}\label{2nddertEgamma}
\exp''_{\gamma,\tau}(t)>\,0\text{ for each }t\ge\,0\,.
\end{equation}
 By \eqref{estimfirstdertildeEgamma}, for each $t\ge\,0$ and $\tau>\,e$, we obtain that
\[
\begin{split}
\left(\log (t+\tau)-\frac{\gamma\,t}{t+\tau}\right)^2
&-(\log (t+\tau))^{\gamma+1}\frac{\gamma\,\tau+ \gamma (t+\tau)}{(t+\tau)^2}+\frac{\gamma (\gamma+1)t\log (t+\tau))^\gamma}{(t+\tau)^2}\\
&\ge\,\left(\log (t+\tau)-\frac{\gamma\,t}{t+\tau}\right)^2
-(\log (t+\tau))^{\gamma+1}\frac{\gamma\,\tau+ \gamma (t+\tau)}{(t+\tau)^2}\\
&\ge\,\left(\log (t+\tau)-\gamma\right)^2-\frac{2\gamma}{\tau}(\log (t+\tau))^{2}\\
&=\log(t+\tau) \left( \log(\tau) -2\gamma \left( \frac{\log(\tau)}{\tau}+1\right) \right) + \gamma^2 \\
&\ge \log(t+\tau) \left( \log(\tau) -2 \left( \frac{\log(\tau)}{\tau}+1\right) \right) .
\end{split}
\]
It is clear that, there is $\tau_0>0$ (independent on $\gamma$), so that the latter quantity is positive for all $\tau\ge\tau_0$ and $t\ge0$.
Hence, \eqref{2nddertEgamma} follows and $\exp_{\gamma,\tau}$ is strictly convex on $[0,\infty)$. 

Let us show \eqref{sublineartEgamma}, that is, for every $t,s\in [0,\infty)$,
\begin{equation}\label{sublineartEgamma2}
\begin{split}
\exp_{\gamma,\tau}(t+s)&=\,\exp\left( \frac{t+s}{(\log (t+s+\tau))^\gamma}\right) \\
&\le\, \exp_{\gamma,\tau}(t)\,\exp_{\gamma,\tau}(s) \\
&=\,\exp\left( \frac{t}{(\log (t+\tau))^\gamma}+\frac{s}{(\log (s+\tau))^\gamma}\right) \,.
\end{split}
\end{equation}
Observe that
\[
\begin{split}
 \frac{t+s}{(\log (t+s+\tau))^\gamma}&=\,\frac{t}{(\log (t+s+\tau))^\gamma}+\frac{s}{(\log (t+s+\tau))^\gamma}\\
 &\le\,\frac{t}{(\log (t+\tau))^\gamma}+\frac{s}{(\log (s+\tau))^\gamma}\,,
 \end{split}
\]
whence \eqref{sublineartEgamma2} follows, being the exponential function nondecreasing.  Inequality \eqref{sublineartE*gamma} follows by using \eqref{sublineartEgamma} and the fact that $\exp_{\gamma,\tau}(t)=\, \exp_{\gamma,\tau}^*(t)+1$.

\noindent(ii) Notice that,
if 
\[
a(t):=\,\exp'_\gamma(t)-\exp'_\gamma(0)\text{ if } t\ge\,0\,,
\]
by \eqref{derEgammapos} and  \eqref{2nddertEgamma},  $a$ is continuous, (strictly) increasing, $a(0)=\,0$, $a(t)>\,0$ if $t>\,0$ and $\lim_{t\to\infty}a(t)=\,\infty$. Moreover, being $a$ increasing, we have for $t\ge0$,
\[
\widetilde \exp_{\gamma,\tau}(t)=\,\exp_{\gamma,\tau}(t)-1-\frac{t}{(\log \tau)^\gamma}=\,\int_0^t\left(\exp'_\gamma(s)-\exp'_\gamma(0)\right)\,ds=\,\int_0^ta(s)\,ds .
\]
Thus, $\widetilde \exp_{\gamma,\tau}$ is (strictly) convex.

Finally, since the functions have a more than linear growth, it is clear that $\widetilde\exp_{\gamma,\tau}(t)/\exp_{\gamma,\tau}(t)$
tends to 1 as $t\to\infty$.
\end{proof}

\begin{remark}
	Notice that, by~\eqref{propAgamma}, the pair $(\widetilde \exp_{\gamma,\tau},\Omega)$ is never $\Delta$-regular for any $\Omega\subset\R^n$.
	In particular, by Lemma~\ref{KAsv}, $K_\vf(\Omega)$ is never a vector space if $\vf= \widetilde \exp_{\gamma,\tau}$.
\end{remark}

\begin{remark}
	Exponential growth functions fall also in the class treated by Theorem~\ref{thm:mainnew}.
	More precisely, the functions	
	\[
	e_\alpha^*(t) := \exp(\alpha t) - 1 ,
	\]
	for $\alpha>0$ satisfy the conditions on $\vf$ given in  Theorem~\ref{thm:mainnew}.
\end{remark}

\section{Proof of the approximation results}

In this section we are going to show our results. 

\begin{proof} [Proof of Theorem~\ref{thm:mainnew}] 
Uniform positivity of $w$ on compact subsets gives 
\begin{equation}\label{eq5ff48a0atgn}
\,\int_R\vf\left(|Db(s,x)|\right)\dd s \dd x< \infty\qquad\text{whenever $R\Subset I\times\Omega$}.
\end{equation}

We are going to exploit an adaptation of the technique of  the proof of Meyers-Serrin's theorem (see, for instance, \cite[Thm.~3.9]{AFP}). 
Let $Q:=\,I\times\Omega$ and let $U_j$, $j=0,1,\ldots$, be the nondecreasing sequence  of open subsets
\[
U_0:=\,\emptyset,\quad U_j:=\,\left\{(s,x)\in Q:\,{\rm dist}((s,x),\partial Q)>\frac{1}{j},\,
|s|+|x| < j
\right\} \,\,\,(j=1,2,\ldots),
\]
and let
\[
Q_j:=\,U_{j+1}\setminus\overline U_{j-1}\quad j=1,2,\ldots\,.
\]
Then $\cup_jQ_j=Q$, each $Q_j$ has compact closure in $Q$ and any point of $Q$ belongs to at most four sets $Q_j$.  
More specifically, if $j\ge\,3$ and $x\in Q_j$, then $x$ may belong at most to $Q_{j-1}$ and $Q_{j+1}$.

Let $(\zeta_j)_j$ be a partition of  unity relative to the covering $(Q_j)$, that is, nonnegative functions  $\zeta_j \in C^\infty_c(Q_j)$ such that $\sum_{j=1}^\infty\zeta_j\equiv\,1$ in $Q$.  Moreover, let $\psi_j\in C^\infty_c(Q)$ be cut-off functions such that $0\le\,\psi_j\le\,1$ in $Q$  
and $\psi_j\equiv \,1$ in $Q_j$. 

For each $j=1,2,\ldots$, let  $b_j:\,\R^{n+1}=\,\R_s\times\rn_x\to\rem$   denote 
 \[
b_j(s,x):=
\begin{cases}
\,\psi_j(s,x)\,b(s,x)
&\text{ if }(s,x)\in I\times\Omega\\
0&\text{ if }(s,x)\in \R^{n+1}\setminus I\times\Omega\,,
\end{cases}
\]
so that it is clear that ${\rm spt}(b_j)\Subset Q$, $b_j\in L^1(\R_s;W^{1,1}(\rn_x;\rem))\cap C^0_{\rm c}(\R^{n+1},\rem)$, 
and
\begin{equation}
Db_j=\,D(\psi_j\,b)=\,\psi_j\,D b+\nabla\psi_j\otimes b\text{ a.e.~in }Q\,.
\end{equation}
In particular, 
\begin{equation}\label{DbtjequalDbtOmegajgn}
b_j=\,b\text{ and } Db_j=\,D b\text{ a.e.~in }Q_j\, ,
\end{equation}
\begin{equation}\label{DbjL1}
Db_j\in L^1(\R^{n+1};\R^{nm})\, .
\end{equation}
The monotonicity of $\vf$ and the weak subadditivity condition \eqref{growthvf} give
\begin{equation}\label{estimvfDbtjgn}
\begin{split}
\vf\left(|Db_j|\right)
&\le\,\vf\left(\psi_j\,|D b|+|\nabla\psi_j\otimes b|\right)\\
&\le\,k_\vf\left[
\vf\left(\psi_j\,|D b|\right)\,\vf\left(|\nabla\psi_j\otimes b|\right)
+\vf\left(\psi_j\,|D b|\right)+\vf\left(|\nabla\psi_j\otimes b|\right) 
\right] .
\end{split}
\end{equation}
Since, by \eqref{eq5ff48a0atgn},
\[
\vf\left(\psi_j\,|D b|\right)\in L^1(\R^{n+1})\quad\text{ and }
\quad\vf\left(|\nabla\psi_j\otimes b|\right)\in L^1(\R^{n+1})\cap L^\infty(\R^{n+1})\,,
\]
we obtain from \eqref{estimvfDbtjgn} that
\begin{equation}\label{tvfDbtjsummablegn}
\vf\left(|Db_j|\right)\in  L^1(\R^{n+1}).
\end{equation}

Let  $(\rho_\eps(s,x))_\eps$ be space-time mollifiers on $\R^{n+1}=\,\R_s\times\rn_x$.
For each $\delta\in (0,1)$  we will make a suitable choice of $0<\eps_j<\delta$ and define 
$b_\delta:\,Q\to\rem$ as 
\[
b_\delta(s,x):= \sum_{j=1}^\infty\zeta_j(s,x)\,(\rho_{\epsilon_j}*b_j)(s,x)\,.
\]
Since the sum is locally finite, $b_\delta$ is well defined.
 Moreover, by construction,  $b_\delta\in C^\infty(I\times\Omega;\rem)$ and
one has
\begin{equation}\label{estimbdelta}
b_\delta\to b\text{ in }L^1_{\rm loc}(Q;\rem)\text{, as } \delta\to 0\,.
\end{equation}
Notice now that
\begin{equation}\label{represDbdeltagn}
\begin{split}
Db_\delta&=\,\sum_{j=1}^\infty D(\zeta_j(\rho_{\eps_j}*b_j))=\,\sum_{j=1}^\infty \zeta_j(\rho_{\eps_j}*Db_j)+\sum_{j=1}^\infty\nabla\zeta_j\otimes(\rho_{\eps_j}* b)\\
&=\,v_\delta+z_\delta\text{ in }Q\,,
\end{split}
\end{equation}
where
\[
v_\delta:=\sum_{j=1}^\infty \zeta_j(\rho_{\eps_j}*Db_j)\,,
\]
and 
\[
z_\delta:=\sum_{j=1}^\infty\left(\nabla\zeta_j\otimes(\rho_{\eps_j}* b_j)-\nabla\zeta_j\otimes b_j\right)\,,
\]
where we used the fact that,
since $\sum_{j=1}^\infty\nabla\zeta_j\equiv 0$, 
we have
$\sum_{j=1}^\infty\nabla\zeta_j\otimes b\equiv 0$.

For each $\delta>\,0$ and $j=1,2,\ldots$, we can find $0<\eps_j<\delta$ such that
\begin{equation}\label{propbde2gn}
\int_Q|\zeta_j\,(\rho_{\eps_j}*b_j)-\zeta_j\,b|\dd x\dd s<\frac{\delta}{2^j}\,
\end{equation}
and
\begin{equation}\label{propbde31newgn}
\|\zeta_j(\rho_{\eps_j}*Db_j)-\zeta_jDb_j\|_{L^1(\R^{n+1};\rem)}
<\frac{\delta}{2^{j+1}}\,.
\end{equation}
In addition, setting 
\[
M_j:=\,\max\{1,\sup_{Q_j}w\} \, ,
\]
we can also ensure that
\begin{equation}\label{propbde3gnnew}
\begin{split}
&\|\nabla\zeta_j\otimes(\rho_{\eps_j}* b_j)-\nabla\zeta_j\otimes b\|_{L^p(Q_j;\R^{nm})}\\
&=\,\|\nabla\zeta_j\otimes(\rho_{\eps_j}* b_j)-\nabla\zeta_j\otimes b\|_{L^p(\R^{n+1};\R^{nm})}<\frac{\delta}{2^{j+1} M_j}\text{ if }p=1,\infty\,
\end{split}
\end{equation}
and
\begin{equation}\label{propbde4newgn}
\|\rho_{\eps_j}*\vf(|Db_j|)-\vf(|Db_j|)\|_{L^1(\R^{n+1})}
<\frac{\delta}{2^j\,M_j}\,.
\end{equation}
Notice that, by assumption on $w$, $M_j<\infty$.  Notice also 
 that the choices in \eqref{propbde2gn} and \eqref{propbde3gnnew} are possible thanks to the continuity of $b$ 
  (in particular, for
 \eqref{propbde3gnnew} with $p=\infty$, we are using also continuity of $b$ with respect to the time variable), while the choices in \eqref{propbde31newgn} and \eqref{propbde4newgn} are possible thanks to the classical properties of convolution together with  \eqref{DbjL1} and \eqref{tvfDbtjsummablegn}, respectively. From \eqref{propbde3gnnew}, it  follows that
 \begin{equation}\label{estimzdeltap=1infty}
 \|z_\delta\|_{L^p(Q;\R^{nm})} < \frac{\delta}{2}\text{ if }p=1,\infty,
 \text{ and}\int_Q |z_\delta| \,w\dd x\dd s <\delta\,.
\end{equation}
Moreover,  by \eqref{estimzdeltap=1infty} with $p=\,\infty$ and   
the convexity of $\vf$, if we set $L:=\vf(1)$, then $\vf(0)=0$ and the monotonicity of difference quotients give
\begin{equation}\label{estimsigmadelta}
\sigma_\delta:=\vf(|z_\delta|)\le\,L |z_\delta|\text{ in }Q\,.
\end{equation}
Notice now that, since $Db=\,\sum_{j=1}^\infty \zeta_j Db=\,\sum_{j=1}^\infty \zeta_j Db_j$,
by  \eqref{propbde31newgn} and \eqref{estimzdeltap=1infty} with $p=1$,
we have
\begin{equation}\label{estimDbdelta}
\begin{split}
\|Db_\delta-Db\|_{L^1(Q;\rem)}&=\,\|v_\delta+z_\delta-Db\|_{L^1(Q;\rem)}\\
&\le\,\|v_\delta-Db\|_{L^1(Q;\rem)}+\,\|z_\delta\|_{L^1(Q;\rem)}\\
&\le\,\sum_{j=1}^\infty\|\zeta_j(\rho_{\eps_j}*Db_j)-\zeta_jDb_j\|_{L^1(\R^{n+1};\rem)}+\frac{\delta}{2}\\
&\le\, \sum_{j=1}^\infty\frac{\delta}{2^{j+1}}+\frac{\delta}{2}=\,\delta\,.
\end{split}
\end{equation}
By \eqref{estimDbdelta}, it follows that
\begin{equation}\label{convDbdelta}
\lim_{\delta\to0} \|Db_\delta-Db\|_{L^1(Q;\rem)} = 0 .
\end{equation} 
In particular, by the continuity of $\vf$ and \eqref{convDbdelta}, there exists an infinitesimal sequence $(\delta_h)_h$ such that, if $b_h:=\,b_{\delta_h}$,
\begin{equation}\label{pointconvfDbh}
\vf\left(|Db_h|\right)\to \vf\left(|Db|\right)\text{a.e. in $Q$, as }h\to\infty\,.
\end{equation}

Let us now show \eqref{convL1mainthm}. One has, a.e.~in~$Q$,
\begin{equation*}
\begin{split}
\vf\left(|Db_{\delta}|\right)&\le \vf\left(|z_\delta|+ |v_\delta|\right)
\le\,k_\vf\left(\vf\left(|z_\delta|\right)\,\vf\left( |v_\delta|\right)+\vf\left(|z_\delta|\right)+\vf\left( |v_\delta|\right)\right)\\
&\le\, k_\vf( (1+\sigma_\delta)\,\vf\left(\left|v_\delta\right|\right)+\sigma_\delta )\,,
\end{split}
\end{equation*}
where $\sigma_\delta=\vf\left(|z_\delta|\right)$.
Set $\sigma_\delta^\infty := \|\sigma_\delta\|_{L^\infty(Q)}$, so that,  a.e.~in~$Q$,
\begin{equation}\label{estimvfDbdelta}
\vf\left(|Db_\delta|\right)
\le\, k_\vf((1+\sigma_\delta^\infty)\,\vf\left(\left|v_\delta\right|\right)+\sigma_\delta)\,.
\end{equation}
By the monotonicity and convexity of $\vf$, by Jensen's inequality and taking into account that $(\zeta_j)_j$ is a partition of 
unity, we get,  a.e.~in~$Q$,
\begin{equation}\label{comparvtdeltaGdeltagn}
\begin{split}
\vf\left(\left|v_\delta\right|\right)&=\,\vf\left(\left|\sum_{j=1}^\infty \zeta_j(\rho_{\eps_j}*Db_j)\right|\right)\le
\,\vf\left(\sum_{j=1}^\infty\zeta_j (\rho_{\eps_j}* |Db_j|)\right)\\
&\le\,\sum_{j=1}^\infty\zeta_j (\rho_{\eps_j}*\vf\left( |Db_j|\right))=:\,
G_\delta\,.
\end{split}
\end{equation}

Hence, by \eqref{estimvfDbdelta} and \eqref{comparvtdeltaGdeltagn}, if follows that
\begin{equation}\label{estimvfDbdelta2}
w\vf\left(|Db_\delta|\right)
\le\,k_\vf ((1+\sigma_\delta^\infty)\,wG_\delta+\, w\sigma_\delta)
\quad \text{a.e. in }Q\,.
\end{equation}
It is clear that, by \eqref{estimzdeltap=1infty} and \eqref{estimsigmadelta}, 
\begin{equation}\label{convsigmadelta}
\text{$w\,\sigma_\delta$ converges to $0$ in $L^1(Q)$ and $\sigma_\delta^\infty\to 0$, as $\delta\to 0$.}
\end{equation}
Let us now prove that
\begin{equation}\label{convGdelta}
\text{$w\,G_{\delta}$ converges to $w\vf(|Db|)$ in $L^1(Q)$, as $\delta\to 0$}\,.
\end{equation}
Observe that, by \eqref{DbtjequalDbtOmegajgn},
\begin{equation*}
\vf(|Db|)=\,\sum_{j=1}^\infty \zeta_j\vf(|Db|)=\,\sum_{j=1}^\infty \zeta_j\vf(|Db_j|)\text{ a.e. in }Q\,,
\end{equation*}
so that \eqref{propbde4newgn} gives \eqref{convGdelta}. 

The combination of \eqref{convGdelta} and \eqref{estimvfDbdelta2} gives the equi-integrability of $w\vf\left(|Db_\delta|\right)$.
By using \eqref{pointconvfDbh}, Vitali's form of the dominated convergence theorem (see, for instance, \cite[Exercise~1.18]{AFP}),
gives \eqref{convL1mainthm}  and the proof is complete.
\end{proof}

 \begin{remark}\label{rem:different_integrability}\label{convNPsi}
 Let $\Psi:[0,\infty)\to [0,\infty)$ be any continuous function with $\Psi(0)=0$ and linear growth at the origin. 
Notice that all the terms  $\zeta_j(\rho_{\epsilon_j}*b_j)-\zeta_jb$ can be made arbitrarily small not only in $L^\infty_t(L^\infty_x)$,
but also in $L^1_t(L^1_x)$, choosing $\epsilon_j\ll 1$. Then, using the representation 
$$(b_{\delta_h}-b)=\sum_{j=1}^\infty\zeta_j(\rho_{\epsilon_j}*b_j)-\zeta_jb$$
we can improve the construction to get also
$$
\lim_{h\to\infty}\int_I\int_\Omega \Psi(|b_{\delta_h}-b|)\dd x\dd s=0.
$$
More precisely, choosing $\epsilon_j\ll 1$ properly, we can make arbitrarily small all terms
$$
\int_{Q_j} \Psi(|\sum_{j=1}^\infty\zeta_j(\rho_{\epsilon_j}*b_j)-\zeta_jb|)\dd x\dd s
$$
since the sum is locally finite and $Q_j\Subset I\times\Omega$.
\end{remark}

%
%

\begin{remark}
	Since the proof of Theorem~\ref{thm:mainnew} is based on a convolution argument,
	we have more control on the convergence depending on the properties of $b$.
	We give two cases that can be of interest.
	
	First, if $b\in L^1(I;C(\Omega;\rem))$ is continuous in the spatial variable, then the approximating sequence $b_h\in C^\infty(I\times\Omega;\rem)$ in Theorem~\ref{thm:mainnew}
	can be taken so that $b_h\to b$ in $b\in L^1(I;C(\Omega;\rem))$.
	
	Second,
	if there exists a bounded open set $\Omega'\Subset\Omega$ such that
	\begin{equation}\label{bcptsupp}
	{\rm spt}(b(t,\cdot))\subset\Omega'\text{ for every }t\in I\,,
	\end{equation}
	then
	the approximating sequence $b_h\in C^\infty(I\times\Omega;\rem)$
	can be taken with 
	\begin{equation}\label{bhcptsupp}
	{\rm spt}(b_h(t,\cdot))\subset \Omega' \text{ for each }t\in I \text{ and }h\in\N\,.
	\end{equation}
\end{remark}

\begin{proof}[Proof of Corollary~\ref{subexpoenergy}] 
Let us  prove that
\begin{equation}\label{sumEgammaDb}
\int_I\int_\Omega w(s,x)\exp_{\gamma,\tau}^*\left(|Db(s,x)|\right)\dd x \dd s<\infty\,.
\end{equation}
Notice that, since $b$ satisfies \eqref{eq5ff48a0aw}
with \eqref{vf=Egamma}, then 
$w\,\widetilde \exp_{\gamma,\tau}\left(|Db|\right)\in L^1(I\times\Omega)$.
On the one hand, if $w|Db|\in L^1(I\times\Omega)$, since 
\[
w\exp_{\gamma,\tau}^*\left(|Db|\right)=\,w\widetilde \exp_{\gamma,\tau}\left(|Db|\right)+\frac{w|Db|}{(\log \tau)^\gamma} 
\]
and $w\,\widetilde \exp_{\gamma,\tau}\left(|Db|\right)\in L^1(I\times\Omega)$
we immediately obtain \eqref{sumEgammaDb}.  
On the other hand, if
$w\in L^1(I\times\Omega)$, by \eqref{propAgamma}, there exists $\bar t>\,0$ such that
\begin{equation}\label{compEgammaAgamma}
\frac{1}{2} \, \exp_{\gamma,\tau}^*(t)\le\, \widetilde \exp_{\gamma,\tau}(t)\text{ for each }t\ge\,\bar t\,.
\end{equation}
Thus
\[
\begin{split}
\int_{I\times\Omega}w\,\exp_{\gamma,\tau}^*\left(|Db|\right)\,\dd s \dd x
&=\,\int_{\{|Db|<\,\bar t\}}w\,\exp_{\gamma,\tau}^*\left(|Db|\right)\,\dd s \dd x \\
&\qquad\qquad +\int_{\{|Db|\ge\,\bar t\}}w\,\exp_{\gamma,\tau}^*\left(|Db|\right)\,\dd s \dd x\\
&\le\,\exp_{\gamma,\tau}^*\left(\bar t\right)\int_{I\times\Omega}w\,\dd s \dd x\\
&\qquad\qquad +2\int_{I\times\Omega}w\,\widetilde \exp_{\gamma,\tau}\left(|Db|\right)\,\dd s \dd x<\,\infty
\end{split}
\]
and \eqref{sumEgammaDb} follows once more. 

By \eqref{sumEgammaDb}, we can apply Theorem~\ref{thm:mainnew} to get the existence of 
$b_h\in C^\infty(I\times\Omega;\rem)$  satisfying \eqref{convW^11loc} and \eqref{convL1mainthm} with $\vf=\,\exp_{\gamma,\tau}^*$. Since
$w\,\widetilde \exp_{\gamma,\tau}\left(|Db_h|\right)\le\,w\,\exp_{\gamma,\tau}^*\left(|Db_h|\right)$,
by applying Vitali's convergence theorem, we obtain again the desired conclusion.
\end{proof}

In the proof of Theorem~\ref{subexpoenergywindepx} we will need the following lemma.
\begin{lemma}\label{lmapproxwx}
Let $f\in L^1_{\rm loc}(\R_s\times\Omega)$  and $ (\rho_\eps(s))_\eps$ be a family of time mollifiers in $\R_s$. Then, for 
a.e. $x\in\Omega$, the time convolution product $f^\eps(\cdot,x):\R_s\to\R$
\begin{equation*}
\begin{split}
f^\eps(s,x)&=\,(\rho_\eps*f(\cdot,x))(s)\\
&:=\int_{\R}\rho_\eps(s-v)\,f(v,x)\,dv\text{ for each }s\in\R
\end{split}
\end{equation*}
and
\begin{equation}\label{lmapproxwx1}
f^\eps(\cdot,x)\in C^0(\R_s)\text{ for each $\eps>\,0$, for a.e. }x\in\Omega\,.
\end{equation}
In addition, for any open set $\omega\subset\Omega$ one has
\begin{equation}\label{lmapproxwx2}
\|f^\eps\|_{L^1(\R_s\times\omega)}\le\,\|f\|_{L^1(\R_s\times\omega)}\text{ for each $\eps>\,0$}
\end{equation}
and
\begin{equation}\label{lmapproxwx3}
f^\eps\to f\text{ in }L^1(\R_s\times\omega)\text{ as }\eps\to 0\,,
\end{equation}
provided that $f\in L^1(\R_s\times\omega)$. 

Finally, if we  assume that, for each ball $B(x_0,r)\Subset\Omega$ one has
\begin{equation}\label{contfinx}
f\in L^1\bigl(\R_s;C^0(B(x_0,r))\bigr)\,,
\end{equation}
then
\begin{equation}\label{contsconv}
f^\eps\in C^0(\R_s\times\Omega)\text{ for each $\eps>\,0$}\,.
\end{equation}
\end{lemma} 
\begin{proof} Properties \eqref{lmapproxwx1},  \eqref{lmapproxwx2} and \eqref{lmapproxwx3} can be proved as in the case of the global $(s,x)$-convolution by mollifiers (see, for instance, \cite[Section~2.1]{AFP}). Let us prove \eqref{contsconv}. Let  $(s_0,x_0)\in\R_s\times\Omega$ and let $((s_h,x_h))_h\subset\R_s\times\Omega$ a sequence converging to $(s_0,x_0)$. From \eqref{contfinx},
\[
\int_\R F(s)\,ds<\infty \text{ if }F(s):=\,\sup_{B(x_0,r)}|f(s,\cdot)|\,,
\]
and, without loss of generality, we can assume that $(x_h)_h\subset B(x_0,r)$ for a fixed $r>0$. Then,
since
\[
 |\rho_\eps(s_h-v)\,f(v,x_h)|\le\,\sup_\R\rho_\eps\,F(v)\quad\text{for a.e. $v\in\R$},
\]
by Lebesgue's dominated convergence theorem, it follows that
\[
\begin{split}
\lim_{h\to\infty}f^\eps(s_h,x_h)&=\,\lim_{h\to\infty}\int_\R \rho_\eps(s_h-v)\,f(v,x_h)\,dv\\
&=\,\int_\R\lim_{h\to\infty}( \rho_\eps(s_h-v)\,f(v,x_h))\,dv\\
&=\,\int_\R \rho_\eps(s_0-v)\,f(v,x_0)\,dv=\,f^\eps(s_0,x_0)\,.
\end{split}
\]
\end{proof}

\begin{proof}[Proof of Theorem \ref{subexpoenergywindepx}] 
We extend $b$ to $\R\times\Omega$ setting $b(t,x)=0$ whenever $t\notin I$. Denoting by $b^\epsilon$ the
mollified functions with respect to the time variable, one has
$$
Db^\epsilon(t,x)=\int_\R\rho_\epsilon(t-s)Db(s)\dd s
$$
and therefore Jensen's inequality gives
\begin{eqnarray*}
\int_\Omega w(x)\vf(Db^\epsilon(t,x))\dd x&\le&\int_\Omega w(x)\int_\R\rho_\epsilon(t-s)\vf(|Db(s,x)|)\dd s\dd x\\
&=&
\int_\R\rho_\epsilon(t-s)\int_\Omega w(x)\vf(Db(s,x)|)\dd x\dd s\,.
\end{eqnarray*}
By integration on $I$, it follows that $N_{\vf,w}(|Db^\epsilon|)\le N_{\vf,w}(|Db|)$.
Notice now that $b^\epsilon$ satisfy the assumptions of Corollary \ref{subexpoenergy}, thanks to \eqref{contsconv}.
Thus for all $\epsilon>0$ we get the existence of a sequence $(b^\epsilon_h)_h\subset C^\infty(I\times\Omega;\rem)$ satisfying \eqref{convW^11loc} and \eqref{convL1mainthm}. Finally, by taking a diagonal sequence, we get the desired conclusion.
\end{proof}

\begin{proof}[Proof of Theorem \ref{thm:main2}]
Notice that for any open set $A\subset\Omega$ the weak $L^1(A;\R^{mn})$ convergence of derivatives grants 
$$
\liminf_{h\to\infty}\int_A \vf\left(|Db_h|\right)\dd x
\ge\int_A \vf\left(|Db|\right)\dd x\, .
$$
Therefore, by applying an elementary lemma (see, for instance, the proof of Proposition~1.80 in \cite{AFP}) 
the convergence of the integrals on $\Omega$ can be localized, getting
$$
\lim_{h\to\infty}\int_A \vf\left(|Db_h|\right)\dd x 
=\int_A \vf\left(|Db|\right)\dd x 
$$
whenever $A\subset\Omega$ is open with Lebesgue negligible boundary.
In particular, choosing $A\Subset\Omega$ with this property, since $A$ has finite measure 
we can use the strict convexity of $\vf$ and \cite{V} to get that $Db_h\to Db$ in $L^1(A;\R^{mn})$.
It follows that $Db_h\to Db$ in  $L^1_{\rm loc}(I\times\Omega;\R^{nm})$ and therefore, modulo the extraction of a subsequence, we can assume that $Db_h\to Db$ a.e. in $\Omega$. 

Combining the pointwise convergence
$$
\lim_{h\to\infty}\vf(|Db_h|)=\vf(|Db|)\quad\text{a.e. in $\Omega$}
$$
with the convergence of the integrals, Scheff\'e's lemma gives that $\vf(|Db_h|)$ converges in $L^1(\Omega)$ to $\vf(|Db|)$. 
Now, the inequality
$$
\vf\left(\frac{|Db_h-Db|}{2}\right)
\le\frac 12  \vf\left(|Db_h|\right)
+\frac 12 \vf\left(|Db|\right)
$$
grants the equi-integrability of $\vf(|Db_h-Db|/2)$. Vitali's convergence theorem can finally be applied to get the result.
\end{proof}

\begin{proof}[Proof of Theorem \ref{approxOS}] Notice that, being $\widetilde\exp$ a $N$-function, it has a linear growth at the origin. Thus by applying Remark~\ref{convNPsi} with $\Psi=\widetilde\exp$  and Corollary ~\ref{subexpoenergy} with $b(t,x)=\,u(x)$ and $w\equiv 1$, we get the desired conclusion.
\end{proof}

\def\refname{References}

\end{document}